\documentclass{amsart}

\usepackage{latexsym,amssymb,amsmath,amsthm,amsopn,graphics,xy,epsfig,picture,epic}%,showkeys}

\usepackage{color}

\def\ff{{\mathcal F}}
\def\ffi{\varphi}
\def\eps{\varepsilon}
\def\dst{\displaystyle}

\def\R{{\mathbb{R}}}

\newcommand{\norm}[1]{{\left\|{#1}\right\|}}
\newcommand{\ent}[1]{{\left[{#1}\right]}}
\newcommand{\abs}[1]{{\left|{#1}\right|}}
\newcommand{\scal}[1]{{\left\langle{#1}\right\rangle}}

\newenvironment{definition}[1][]{\vskip3pt\noindent\sl\textbf{Definition.}\ }{\rm\vskip3pt}
\newenvironment{remark}[1][]{\vskip3pt\noindent\textbf{Remark.}\ }{\rm\vskip3pt}

\newtheorem{lemma}{Lemma}[section]
\newtheorem{proposition}[lemma]{Proposition}
\newtheorem{theorem}[lemma]{Theorem}
\newtheorem{corollary}[lemma]{Corollary}

\date{\today}

\newcounter{rep}
\newcommand{\rep}[1]{
		}

\newcounter{rea}
\newcommand{\rea}[1]{
		}

\newcounter{rer}
\newcounter{res}
\begin{document}

\title[Approximation of functions by finite Hermite series]{Approximation of almost time and band limited functions I:
Hermite expansions}

\author{Philippe Jaming, Abderrazek Karoui, Ron Kerman, Susanna Spektor}

\address{Philippe Jaming
\noindent Address: Institut de Math\'ematiques de Bordeaux UMR 5251,
Universit\'e Bordeaux 1, cours de la Lib\'eration, F 33405 Talence cedex, France}
\email{Philippe.Jaming@gmail.com}

\address{Abderrazek Karoui
\noindent Address: Universit\'e de Carthage,
D\'epartement de Math\'ematiques, Facult\'e des Sciences de
Bizerte, Tunisie.}
\email{Abderrazek.Karoui@fsb.rnu.tn}

\address{Ron Kerman
\noindent Address: Brock University, St. Catharines, Canada}
\email{rkerman@brocku.ca}

\address{Susanna Spektor
\noindent Address: University of Alberta, CAB 632, Edmonton, Alberta, Canada }
\email{sanaspek@gmail.com}

\begin{abstract}
The aim of this paper is to investigate the quality of approximation of almost time and band limited
functions by its expansion in the Hermite and scaled Hermite basis. As a corollary, this allows us to obtain the rate of convergence
of the Hermite expansion of function in the $L^2$-Sobolev space with fixed compact support.
\end{abstract}

%\thanks{
%This work was supported in part by the  ANR grant "AHPI" ANR-07-
%BLAN-0247-01, the French-Tunisian  CMCU 10G 1503 project and the
%DGRST  research grant 05UR 15-02.\\
%Part of this work was done while the first and second author were visiting
%and the third author was staff of the research laboratory MAPMO of the University of Orl\'eans,
%France.}


%, prolate spheroidal wave functions.}

\maketitle

\section{Introduction}

The aim of this paper is to investigate the quality of approximation of almost time and band limited
functions by its expansion in the Hermite basis. As a corollary, this allows us to obtain the rate of convergence
of the Hermite expansion of function in the $L^2$-Sobolev space with fixed compact support.

\smallskip

Time-limited functions and band-limited functions play a fundamental role in signal and image processing.
The time-limiting assumption is natural as a signal can only be  measured over a finite duration.
The band-limiting assumption is natural as well due to channel capacity limitations. It is also
essential to apply sampling theory. Unfortunately, the simplest form of the uncertainty principle
tells us that a signal can not be simultaneously time and band limited.
A natural assumption is thus that a signal is almost time and band limited in the following sense:

\begin{definition} Let $T,\Omega>0$ and $\eps_T,\eps_\Omega>0$. A function $f\in L^2(\R)$
is said to be 
\begin{itemize}
\item $\eps_T$-\emph{almost time limited} to $[-T,T]$ if
$$
\int_{|t|>T}|f(t)|^2\,\mathrm{d}t\leq\eps_T^2\norm{f}_{L^2(\R)}^2;
$$
\item $\eps_\Omega$-\emph{almost band limited} to $[-\Omega,\Omega]$ if
$$
\int_{|\omega|>\Omega}|\widehat{f}(\omega)|^2\,\mathrm{d}\omega\leq\eps_\Omega^2\norm{f}_{L^2(\R)}^2.
$$
\end{itemize}
Here and throughout this paper the Fourier transform is normalized so that, for $f\in L^1(\R)$,
$$
\widehat{f}(\omega):=\ff[f](\omega):=\frac{1}{\sqrt{2\pi}}\int_{\R}f(t)e^{-it\omega}\,\mathrm{d}t.
$$
\end{definition}

Of course, given $f\in L^2(\R)$, for every $\eps_T,\eps_\Omega>0$ there exist $T,\Omega>0$ such that
$f$ is $\eps_T$-almost time limited to $[-T,T]$ and $\eps_\Omega$-almost time limited to $[-\Omega,\Omega]$.
The point here is that we consider $T,\Omega,\eps_T,\eps_\Omega$ as fixed parameters.
A typical example we have in mind is that $f\in H^s(\R)$ and is time-limited to $[-T,T]$.
Such  an hypothesis is common in tomography, {\it see e.g.} \cite{Nat},
where it is required in the proof of the convergence of the filtered back-projection algorithm for approximate
inversion of the Radon transform.
But, if $f\in H^s(\R)$ with $s>0$, that is if 
$$
\norm{f}_{H^s(\R)}^2:=\int_{\R}(1+|\omega|)^{2s}|\widehat{f}(\omega)|^2\,\mbox{d}\omega<+\infty,
$$
then 
\begin{eqnarray*}
\int_{|\omega|>\Omega}|\widehat{f}(\omega)|^2\,\mbox{d}\omega&\leq&
\int_{|\omega|>\Omega}\frac{(1+|\omega|)^{2s}}{(1+|\Omega|)^{2s}}|\widehat{f}(\omega)|^2\,\mbox{d}\omega\\
&\leq&\frac{\norm{f}_{H^s(\R)}^2}{(1+|\Omega|)^{2s}}.
\end{eqnarray*}
Thus $f$ is $\dst\frac{1}{(1+|\Omega|)^{s}}\frac{\norm{f}_{H^s}}{\norm{f}_{L^2(\R)}}$-almost band limited to
$[-\Omega,\Omega]$.

An alternative to the back projection algorithms in tomography are the Algebraic Reconstruction Techniques
(that is variants of Kaczmarz algorithm, {\it see} \cite{Nat}). For those algorithms to work well it is crucial to have a good
representing system (basis, frame...) of the functions that one wants to reconstruct. Thanks to the seminal
work of Landau, Pollak and Slepian, the optimal orthogonal system for representing
almost time and band limited functions is known.
The system in questions consists of the so called prolate spheroidal wave functions $\psi_k^T$
and has many valuable properties (see \cite{prolate1,prolate2,prolate3,prolate4}). Among the most striking properties they have
is that, if a function is almost time limited to $[-T,T]$ and almost band limited to $[-\Omega,\Omega]$ 
then it is well approximated by its projection on the first $4\Omega T$ terms of the basis:
\rea{with these assumptions, it is $4\Omega T$ and not $8\Omega T$}
\begin{equation}
\label{eq:prolate}
f\simeq\sum_{0\leq k<4\Omega T}\scal{f,\psi_k^T}\psi_k^T.
\end{equation}
This is a remarkable fact as this is exactly the heuristics given by Shannon's sampling formula
\rea{added reference LP2}
(note that to make this heuristics clearer, the functions are usually almost time-limited to $[-T/2,T/2]$
and this result is then known as the $2\Omega T$-theorem, see \cite{prolate2}). 

However, there is a major difficulty with prolate spheroidal wave functions that has attracted  a lot of interest recently,
namely the difficulty to compute them as there is no inductive nor closed form formula (see e.g. \cite{BKnote,BK1,Boyd1,Li,Xiao}). One approach
is to explicitly compute the coefficients of the prolate spheroidal wave functions in terms of a basis of orthogonal
polynomials like the Legendre polynomials or in the Hermite basis. The question that then arises
is that of directly approximating almost time and band limited functions by the (truncation of) their expansion
in the Hermite basis. This is the question we address here. We postpone the same question concerning Legendre
polynomials for which we use different methods.

An other motivation for this work comes from the work of the first author \cite{JP} on uncertainty principles for orthonormal bases.
There it is shown that an orthonormal basis $(e_k)$ of $L^2(\R)$
can not have uniform time-frequency localization. Several ways of measuring
localization were considered, and for most of them, the Hermite functions provided the optimal behavior.
However, in one case, the proof relied on \eqref{eq:prolate}: this shows that the set of functions
that are $\eps_T$-time limited to $[-T,T]$ and $\eps_\Omega$-band limited to $[-\Omega,\Omega]$
is almost of dimension $8\Omega T$. In particular, this set can not contain more than
a fixed number of elements of an orthonormal sequence. As this proof shows, the optimal basis
here consists of prolate spheroidal wave functions. As the Hermite basis is optimal
for many uncertainty principles, it is thus natural to ask how far it is from optimal in this case.

Let us now be more precise and describe the main results of the paper.

\medskip

Recall that the Hermite basis $(h_k)_{k\geq 0}$ is an orthonormal basis of $L^2(\R)$ given by
$h_k=\alpha_k\dst e^{x^2/2}\partial^k e^{-x^2}$ where $\alpha_k$ is a normalization constant.
Recall also that the $h_k$'s are eigenfunction of the Fourier transform. Morover,
as is well known the $h_k$'s satisfy a second order differential equation. This allows us
to use the standard WKB method to approximate the Hermite functions as follows: let $\lambda=\sqrt{2n+1}$,
$p(x)=\sqrt{\lambda^2-x^2}$ and $\dst\ffi(x)=\int_0^xp(t)\,\mbox{d}t$,
then, for $|x|<\lambda$,
\begin{equation}
\label{eq:intro1}
h_k(x)=h_n(0)\sqrt{\frac{\lambda}{p(x)}}\cos\ffi(x)+h_n^\prime(0)\frac{\sin\ffi(x)}{\sqrt{\lambda p(x)}}+error.
\end{equation}
This formula is not new (e.g. \cite{Do,kochtataru,larsson}).
However, we will need a precise estimate of the error term, both in the $L^\infty$ sense for which we improve
the one given in \cite{BKH} and the Lipschitz bound. 

A first consequence of this formula is that the $L^2$-mass of $h_n$ is essentially concentrated in an annulus
of radius $\sqrt{2n+1}$ and width $\leq 1$ of the time-frequency plane. 
A second consequence is the approximaion over $[-T,T]\times[-T,T]$ of the kernel
$$k_n(x,y)=\sum_{k=0}^n h_k(x) h_k(y).$$
\rea{A more precise statement of the approximation of $k_n(x,y)$}
More precisely, by using  \eqref{eq:intro1} and the Christoffel-Darboux formula, one gets for  $n\geq 2T^2$:
\begin{equation}
\label{eq:intro2}
k_n(x,y)=\frac{1}{\pi}\frac{\sin N(x-y)}{x-y}+R_n(x,y),
\end{equation}
where
$$N=\frac{\sqrt{2n+1}+\sqrt{2n+3}}{2},\quad
|R_n(x,y)|\leq \frac{17 T^2}{\sqrt{2n+1}}.$$
Again, this approximation is not new \cite{Sa,Us} but we improve the error estimates. Nonetheless, from
numerical evidences, our previous theoretical   error estimate is still far from the actual error.
\rea{We move here the definition of the operator $\mathcal R_n.$}
Next, let $\mathcal R_n^T$ be the Hilbert-Schmidt operator defined on $L^2([-T,T])$ by
\begin{equation}
\label{integral_operator}
\mathcal R_n^T f(x)= \int_{[-T,T]} R_n(x,y) f(y)\, dy.
\end{equation}
The heuristic is then as follows. Assume that $f$ is $(\eps_T,\eps_{\Omega})$ time and band limited in $[-T,T]\times[-\Omega,\Omega]$.
Thus $f$ is only ``correlated'' to the first $\sim N:=\max(T^2,\Omega^2)$ Hermite functions\,:
$|\scal{f,h_k}|$ is small if $k>N$.
\rea{A more precise statement of the error}
One may thus expect that $f=\sum_{0\leq k \leq N}\scal{f,h_k}h_k+error,$ where the error has 
a satisfactory decay rate with respect to $N.$
This seems unfortunately not to be the case. We establish that for $n\geq N$, the $error=f-\sum_{0\leq k \leq n}\scal{f,h_k}h_k$
has an $L^2$-norm bounded by $\lesssim T^3/\sqrt{2n+1} +\eps_T+\eps_{\Omega}$:

\begin{theorem}\label{th:mainhermite}
Let $\Omega_0,T_0\geq2$ and $\eps_T,\eps_\Omega>0$. Assume that
$$
\int_{|t|>T_0}|f(t)|^2\,\mathrm{d}t\leq \eps_T^2\norm{f}_{L^2(\R)}^2
\quad\mbox{and}\quad
\int_{|\omega|>\Omega_0}|\widehat{f}(\omega)|^2\,\mathrm{d}\omega\leq \eps_\Omega^2\norm{f}_{L^2(\R)}^2.
$$
For $n$ an integer, let $K_nf$ be the orthogonal projection of $f$
on the span of $h_0,\ldots,h_n$.

Assume that $n\geq \max(2T^2,2\Omega^2)$. Then, for $T\geq T_0$,
\begin{equation}
\label{eq:approxhermlocintro}
\norm{f-K_nf}_{L^2([-T,T])}
\leq\left(\eps_\Omega+\frac{34T^3}{\sqrt{2n+1}}+2\eps_T\right)\norm{f}_{L^2(\R)}
\end{equation}
\end{theorem}

In particular, on would need $\sim T^6/\eps^2$ terms to reach an error $\lesssim\eps$. The above heuristics
suggest that the right power of $T$ may \eqref{eq:approxhermlocintro} should be closer to $1$.
\rep{Rephrased sentence on scaling}
We will show how one can decrease the dependence on $T$ by replacing the Hermite basis by a scaled version
$h_k^\alpha(x)=\alpha^{1/2}h_n(\alpha x)$ at the expense of a worse dependence on the almost band-limitness of $f$.

%In order to prove this result, we approximate the projection on the $n$ first Hermite functions
%by the time-frequency localization operator: as is well known, this projection
%is given by an integral kernel:
%$$
%K_nf(x):=\sum_{k=0}^n\scal{h_k,f}h_k(x)=\int_{\R}k_n(x,y)f(y)\,\mbox{d}y.
%$$
\rea{A  more realistic error approximation}\rep{removed this from introduction}

\medskip

The remaining of this paper is organized as follows. The next section is devoted to the approximation
of Hermite functions by the WKB method. We then devote Section \ref{sec:kernel} and \ref{sec:kernel2} to
establish properties of the kernel of the projection on the Hermite functions. In Section \ref{sec:mainth}
we first prove Theorem \ref{th:mainhermite}. Then, we give the  quality of approximation of almost time and band limited functions by the scaled Hermite functions. 
Finally, in the last section, we give various numerical 
examples that illustrate the different results of this work.

\section{Approximating Hermite functions with the WKB method}\label{sec:wkb}

\subsection{The WKB method}

Let $H_n$ be the $n$-th Hermite polynomial, that is
$$
H_n(x)=e^{x^2}\frac{\mathrm{d}^n}{\mathrm{d}x^n}e^{-x^2}.
$$
Define the Hermite functions as
$$
h_n(x)=\alpha_nH_n(x)e^{-x^2/2}\quad\mbox{where }\alpha_n=\frac{1}{\pi^{1/4}\sqrt{2^nn!}}.
$$
As is well known:
\begin{enumerate}
\renewcommand{\theenumi}{\roman{enumi}}
\item $(h_n)_{n\geq 0}$ is an orthonormal basis of $L^2(\R)$.
\item $h_n$ is even if $n$ is even and odd if $n$ is odd, in particular
$h_{2p}^{\prime}(0)=0$ and $h_{2p+1}(0)=0$.
\item $\dst h_{2p}(0)=\frac{(-1)^p}{\pi^{1/4}}\sqrt{\frac{(2p-1)!!}{(2p)!!}}=\frac{(-1)^p}{\sqrt{\pi}p^{1/4}}\left(1-\frac{\eta_{2p}}{8p}\right)$ with $0<\eta_{2p}<1$.

\item  $\dst h_{2p+1}^{\prime}(0)=\frac{(-1)^p\sqrt{4p+2}}{\pi^{1/4}}\sqrt{\frac{(2p-1)!!}{(2p)!!}}
=\frac{(-1)^p\sqrt{4p+3}}{\sqrt{\pi}p^{1/4}}\left(1-\frac{\eta_{2p+1}}{4p}\right)$\\
with $|\eta_{2p+1}|<1$.

\item $h_n$ satisfies the differential equation
\begin{equation}
\label{eq:diffherm}
h_n^{\prime\prime}(x)+(2n+1-x^2)h_n(x)=0.
\end{equation}
\end{enumerate}

We will now follow the WKB method to obtain an approximation of $h_n$.

In order to simplify notation, we will fix $n$ and drop all supscripts during the computation.
Let $h=h_n$, $\lambda=\sqrt{2n+1}$, and define $p(x)$ and $\ffi(x)$ for $|x|<\lambda$ as
$$
p(x)=\sqrt{\lambda^2-x^2}\quad\mbox{and}\quad\ffi(x)=\int_0^xp(t)\,\mbox{d}t.
$$
Note that \eqref{eq:diffherm} reads $h''(x)+p(x)h(x)=0$. Let us define
$$
\psi_\pm(x)=\frac{1}{\sqrt{p(x)}}\exp\pm i\ffi(x)\qquad|x|<\lambda.
$$
\begin{remark}
These functions are introduced according to the standard WKB method. The factor $\exp\pm i\ffi(x)$
would be the solution of \eqref{eq:diffherm} if $p$ where a constant. The factor $p^{-1/2}$ is here to make
the wronskian of $\psi_+,\psi_-$ constant. Indeed, as $\ffi'=p$, a simple computation shows that
\begin{equation}
\label{eq:psiprime}
\psi_\pm^\prime(x)=\left(-\frac{1}{2}\frac{p'(x)}{p(x)}\pm ip(x)\right)\psi_\pm(x).
\end{equation}
It follows that
\begin{eqnarray}
\psi_+(x)\psi_-^\prime(x)-\psi_-(x)\psi_+^\prime(x)&=&\psi_+(x)\psi_-(x)\abs{\begin{matrix}1&1\\
-\frac{1}{2}\frac{p'(x)}{p(x)}+ ip(x)&-\frac{1}{2}\frac{p'(x)}{p(x)}- ip(x)\end{matrix}}\nonumber\\
&=&\frac{1}{p(x)}\times\bigl(-2ip(x)\bigr)=-2i.
\label{eq:wronskian}
\end{eqnarray}
\end{remark}
Using \eqref{eq:psiprime} it is not hard to see that $\psi_\pm$ both satisfy the differential equation
\begin{equation}
\label{eq:diffpsi}
y''+(p^2-q)y=0\qquad\mbox{where }q=\frac{1}{2}\left(\frac{p'}{p}\right)'-\frac{1}{4}\left(\frac{p'}{p}\right)^2.
\end{equation}
A simple computation shows that $q(x)=-\dst\frac{2\lambda^2+3x^2}{4p(x)^4}$.
%%
%% $p(x)=\sqrt{\lambda^2-x^2}$, $p'(x)=\frac{-x}{p(x)}$, $\frac{p'}{p}=-\frac{x}{p^2}$
%% thus
%% $$
%% \left(\frac{p'}{p}\right)'=-\frac{1}{p^2}+2\frac{xp'}{p^3}=-\frac{p^2+2x^2}{p^4}=-\frac{\lambda^2+x^2}{p^4}
%% $$
%% Thus
%% $$
%% q(x)=-\frac{2\lambda^2+2x^2+x^2}{4p^4}=\dst-\frac{2\lambda^2+3x^2}{4p^4}.
%% $$
%%
We will frequently use that $|q(x)|\leq\dst\frac{5\lambda^2}{4p(x)^4}$.
Note also that, if $0<\eta<1$ and $|x|\leq\lambda(1-\eta)$ then $p(x)\geq\lambda\sqrt{2\eta-\eta^2}\geq \lambda\sqrt{\eta}$
while $q(x)\leq\dst\frac{5}{4\lambda^2\eta^2}$.

\smallskip
Now multiplying \eqref{eq:diffpsi} by $h$, \eqref{eq:diffherm} by $\psi_\pm$ and substracting both results, we obtain
$$
h''\psi_\pm-\psi_\pm^{\prime\prime}h+qh\psi_\pm=0.
$$
On the other hand, $h''\psi_\pm-\psi_\pm^{\prime\prime}h=(h'\psi_\pm-\psi_\pm^{\prime}h)'$. Therefore,
\begin{equation}
\label{eq:wkbdiff}
(h'\psi_\pm-\psi_\pm^{\prime}h)'=-qh\psi_\pm.
\end{equation}
Let us now define
$$
Q_\pm(x)=\int_0^xq(t)h(t)\psi_\pm(t)\,\mbox{d}t.
$$
Integrating \eqref{eq:wkbdiff} between $0$ and $x$, we obtain the system
$$
\left\{\begin{matrix}
h'(x)\psi_+(x)&-&h(x)\psi_+^{\prime}(x)&=&h'(0)\psi_+(0)-h(0)\psi_+^{\prime}(0)&-&Q_+(x)\\
h'(x)\psi_-(x)&-&h(x)\psi_-^{\prime}(x)&=&h'(0)\psi_-(0)-h(0)\psi_-^{\prime}(0)&-&Q_-(x)\\
\end{matrix}\right..
$$
According to \eqref{eq:wronskian} the determinant of this system is $-2i$, we can thus solve it for $h(x)$. This leads to
\begin{eqnarray*}
h(x)&=&h(0)\frac{\psi_+^\prime(0)\psi_-(x)-\psi_-^\prime(0)\psi_+(x)}{2i}\\
&&+h'(0)\frac{\psi_-(0)\psi_+(x)-\psi_+(0)\psi_-(x)}{2i}\\
&&+\frac{Q_+(x)\psi_-(x)-Q_-(x)\psi_+(x)}{2i}.
\end{eqnarray*}
It remains to identify those 3 terms. First, note that
$\psi_+(0)=\psi_-(0)=1/\sqrt{p(0)}=1/\sqrt{\lambda}$
while $\dst\psi_+^{\prime}(0)=\overline{\psi_-^{\prime}(0)}=\left(-\frac{1}{2}\frac{p'(0)}{p(0)}+ip(0)\right)\psi_+(0)=i\sqrt{\lambda}$.
From this, we get
\begin{eqnarray*}
\frac{\psi_+^\prime(0)\psi_-(x)-\psi_-^\prime(0)\psi_+(x)}{2i}&=&\sqrt{\frac{\lambda}{p(x)}}\frac{e^{i\ffi(x)}+e^{-i\ffi(x)}}{2}\\
&=&\sqrt{\frac{\lambda}{p(x)}}\cos\ffi(x).
\end{eqnarray*}
Further,
\begin{eqnarray*}
\frac{\psi_-(0)\psi_+(x)-\psi_+(0)\psi_-(x)}{2i}&=&
\frac{1}{\sqrt{\lambda p(x)}}\frac{e^{i\ffi(x)}-e^{-i\ffi(x)}}{2i}\\
&=&\frac{1}{\sqrt{\lambda p(x)}}\sin\ffi(x).
\end{eqnarray*}
Finally,
\begin{eqnarray*}
\frac{Q_+(x)\psi_-(x)-Q_-(x)\psi_+(x)}{2i}&=&\frac{1}{\sqrt{p(x)}}\int_0^x\frac{q(t)}{\sqrt{p(t)}}h(t)\frac{e^{i\ffi(t)}e^{-i\ffi(x)}
-e^{-i\ffi(t)}e^{i\ffi(x)}}{2i}\,\mbox{d}t\\
&=&\frac{1}{\sqrt{p(x)}}\int_0^x\frac{q(t)}{\sqrt{p(t)}}h(t)\sin\bigl(\ffi(t)-\ffi(x)\bigr)\,\mbox{d}t.
\end{eqnarray*}

We are now in position to prove the following theorem:

\begin{theorem}
\label{th:approxherm}
Let $n\geq 0$. Assume that $|x|\leq\sqrt{2n+1}$, then
\begin{eqnarray}
h_n(x)&=&h_n(0)\left(\frac{2n+1}{2n+1-x^2}\right)^{1/4}\cos\ffi_n(x)+h_n^{\prime}(0)
\frac{\sin\ffi_n(x)}{\bigl((2n+1)(2n+1-x^2)\bigr)^{1/4}}\notag\\
&&+E_n(x)
\label{eq:approxherm}
\end{eqnarray}
where
\begin{equation}
\label{eq:approxhermest}
\ffi_n(x)=\int_0^x\sqrt{2n+1-t^2}\,\mbox{d}t
\quad\mbox{and}\quad
|E_n(x)|\leq\frac{5}{4}\left(\frac{\sqrt{2n+1}}{2n+1-x^2}\right)^{5/2}.
\end{equation}
Moreover, for $(2n+1)^{-a}<\eta<1$, $a<3/20$ and $x,y\leq(1-\eta)\sqrt{2n+1}$
\begin{equation}
\label{eq:LipestE}
|E_n(x)-E_n(y)|\leq\frac{5}{(2n+1)^{3/4-5a}}|x-y|.
\end{equation}
Further, if $|x|,|y|\leq T\leq\frac{\sqrt{2n+1}}{2}$,
$$
\ffi_n(x)=\sqrt{2n+1}x-e_n(x),
$$
where
\begin{equation}
\label{eq:esten}
|e_n(x)|\leq \frac{T^3}{3\sqrt{2n+1}}
\quad\mbox{and}\quad
|e_n(x)-e_n(y)|\leq \frac{T^2}{\sqrt{2n+1}}|x-y|,
\end{equation}
while
\begin{equation}
\label{eq:estEn}
|E_n(x)|\leq\frac{2}{(2n+1)^{3/2}}
\quad\mbox{and}\quad |E_n(x)-E_n(y)|\leq\frac{8}{(2n+1)^{5/4}}|x-y|.
\end{equation}
\end{theorem}

\begin{remark}
One may explicitly compute $\ffi$:
$$
\ffi_n(x)=\frac{2n+1}{2}\arcsin\frac{x}{\sqrt{2n+1}}+\frac{x}{2}\sqrt{2n+1-x^2}=\sqrt{2n+1}x-e(x),
$$
where
$$
e_n(x)=\frac{1}{2}\ent{(2n+1)\left(\frac{x}{\sqrt{2n+1}}-\arcsin\frac{x}{\sqrt{2n+1}}\right)
+x\left(\sqrt{2n+1}-\sqrt{2n+1-x^2}\right)}.
$$
Also, $\ffi_n$ has a geometric interpretation: it this the area of the intersection of a disc of radius $\sqrt{2n+1}$
centered at $0$ with the strip $[0,x]\times\R^+$. In particular, when $x\to\sqrt{2n+1}$, $\ffi_n(x)\sim \frac{\pi}{4}(2n+1)$.
\end{remark}

\begin{proof} We will fix $n$ and use the same notation as previously, e.g.
$\lambda=\sqrt{2n+1}$, $p(x)=\sqrt{\lambda^2-x^2}$,...

Let us first establish the bounds on $e$. Note that
\begin{eqnarray*}
e(x)&=&\int_0^x\lambda-\sqrt{\lambda^2-t^2}\,\mbox{d}t
=\lambda\int_0^x1-\sqrt{1-(t/\lambda)^2}\,\mbox{d}t\\
&=&\lambda^2\int_0^{x/\lambda}1-\sqrt{1-s^2}\,\mbox{d}s=\lambda^2\int_0^{x/\lambda}\frac{s^2}{1+\sqrt{1-s^2}}\,\mbox{d}s.
\end{eqnarray*}
But, $\dst \abs{\int_a^b\frac{s^2}{1+\sqrt{1-s^2}}\,\mbox{d}s}\leq\abs{\int_a^bs^2\,\mbox{d}s}=\frac{|b^3-a^3|}{3}$
the estimate of $e(x)$ and $e(x)-e(y)$ follow immediately.

\medskip

Consider
$$
E(x)=\frac{1}{\sqrt{p(x)}}\int_0^x\frac{q(t)}{\sqrt{p(t)}}h(t)\sin\bigl(\ffi(x)-\ffi(t)\bigr)\,\mbox{d}t.
$$
Using Cauchy-Schwarz, we obtain
$$
|E(x)|\leq \frac{1}{\sqrt{p(x)}}\left(\int_0^x\frac{q(t)^2}{p(t)}\,\mbox{d}t\right)^{1/2}
\left(\int_0^x h(t)^2\,\mbox{d}t\right)^{1/2}
\leq
\frac{1}{\sqrt{p(x)}}\left(\int_0^x\frac{25\lambda^4}{16p(t)^9}\,\mbox{d}t\right)^{1/2}
$$
since $\norm{h_n}_2=1$.
As $|x|<\lambda$, and $p$ decreases, the estime $|E(x)|\leq \dst\frac{5\lambda^{5/2}}{4p(x)^5}$ follows.

Note that, if $|x|\leq\lambda/2$, then a slightly better estimate holds:
$$
|E(x)|\leq \frac{10}{4\lambda\sqrt{3}}\left(\int_0^{\lambda/2}\frac{\lambda^4}{(\lambda^2-t^2)^{9/2}}\,\mbox{d}t\right)^{1/2}
=\frac{10}{4\sqrt{3}\lambda^{3}}\left(\int_0^{1/2}\frac{1}{(1-s^2)^{9/2}}\,\mbox{d}s\right)^{1/2}.
$$
A numerical computation shows that $|E(x)|\leq \frac{2}{\lambda^3}$.

\begin{remark}
Note that the bound on $E$ allows to obtain a bound on $h_n$. For instance, if $n\geq2$ is even
\begin{eqnarray*}
|h_{2n}(x)|&\leq&\sqrt{\frac{\lambda}{p(x)}}|h_{2n}(0)|+ \frac{5}{4}\left(\frac{\lambda^{1/2}}{p(x)}\right)^5\\
&\leq&\left(\frac{(2n+1)^{1/4}}{\sqrt{\pi}n^{1/4}}+\frac{5}{4}\frac{\lambda^{5/2}}{p(x)^{9/2}}\right)\frac{1}{\sqrt{p(x)}}\\
&\leq&\left(\frac{(2n+1)^{1/4}}{\sqrt{\pi}n^{1/4}}+\frac{5}{4}\frac{1}{\lambda^{2}\eta^{9/4}}\right)\frac{1}{\sqrt{p(x)}}
\leq\frac{1}{\sqrt{p(x)}},
\end{eqnarray*}
provided $|x|\leq(1-\eta)\lambda$ with $\eta\geq\dst\frac{2}{\lambda^{8/9}}$.

The same estimate is valid in the case when $n$ is odd.
\end{remark}

In order to prove the Lipschitz bound on $E$, let us introduce some further notation:
$$
\chi(x,t)=\frac{q(t)}{\sqrt{p(t)}}h(t)\sin\bigl(\ffi(x)-\ffi(t)\bigr)
$$
and
$$
\Phi(x,y)=\int_0^x\chi(y,t)\,\mbox{d}t.
$$
Thus, we have proved that for $|x|<\lambda$,
$$
|\Phi(x,x)|=|E(x)|\leq\begin{cases}\frac{5}{4}\left(\frac{\lambda^{1/2}}{p(x)}\right)^5&\mbox{if }|x|<\lambda\\
\frac{2}{\lambda^3}&\mbox{if }|x|<\lambda/2\end{cases}.
$$
Now, if $x\leq y<\lambda$,
\begin{eqnarray*}
E(y)-E(x)&=&\left(\frac{1}{\sqrt{p(y)}}-\frac{1}{\sqrt{p(x)}}\right)\Phi(y,y)\\
&&+\frac{1}{\sqrt{p(x)}}\bigl[\Phi(y,y)-\Phi(x,y)\bigr]\\
&&\frac{1}{\sqrt{p(x)}}\bigl[ \Phi(x,y)-\Phi(x,x)\bigr]\\
&=&E_1+E_2+E_3.
\end{eqnarray*}
Note, that
\begin{eqnarray*}
\abs{\frac{1}{\sqrt{p(y)}}-\frac{1}{\sqrt{p(x)}}}&\leq&\frac{1}{2}|x-y|\sup_{t\in[x,y]}\abs{\frac{p'(t)}{p(t)^{3/2}}}\\
&=&\frac{1}{2}|x-y|\sup_{t\in[x,y]}\abs{\frac{t}{p(t)^{5/2}}}\leq\frac{\lambda}{2p(y)^{5/2}}|x-y|.
\end{eqnarray*}
Thus, we obtain that $|E_1|\leq\dst\frac{5\lambda^{7/2}}{8p(y)^{15/2}}|x-y|$.

In the case when $|x|,|y|<\lambda/2$, the same reasoning leads to the estimate $|E_1|\leq\dst\frac{|x-y|}{\lambda^{9/2}}$.

\smallskip

Next, if $|x|,|y|\leq(1-\eta)\lambda$ one can estimate $E_2$ as follows:
\begin{eqnarray*}
|\Phi(y,y)-\Phi(x,y)|&\leq&\int_x^y|\chi(y,t)|\,\mbox{d}t
\leq|x-y|\sup_{t\in[x,y]}\frac{q(t)}{\sqrt{p(t)}}\sup_{|t|\leq|y|\lambda}|h(t)|\\
&\leq& \frac{5\lambda^2}{4p(y)^5}|x-y|.
\end{eqnarray*}
Therefore, $|E_2|\leq \dst\frac{5\lambda^2}{4p(y)^{11/2}}|x-y|$.

In general, we will bootstrap the approximation of $h$. Let us first assume that $n$ is even, so that
$$
h(t)=h_n(0)\sqrt{\frac{\lambda}{p(x)}}\cos\ffi(x)+E(x)
$$
Then
\begin{eqnarray*}
\chi(x,t)&=&\frac{q(t)}{\sqrt{p(t)}}h(t)\sin\bigl(\ffi(x)-\ffi(t)\bigr)\\
&=&h(0)\sqrt{\lambda}\frac{q(t)}{p(t)}\cos\ffi(x)\sin\bigl(\ffi(x)-\ffi(t)\bigr)
+\frac{q(t)}{\sqrt{p(t)}}E(t)\sin\bigl(\ffi(x)-\ffi(t)\bigr)\\
&=&\chi_1(x,t)+\chi_2(x,t).
\end{eqnarray*}
Therefore, we may write $|E_2|\leq E_2^1+E_2^2$ where $\dst E_2^j=\frac{1}{\sqrt{p(x)}}\int_x^y|\chi_j(y,t)|\,\mbox{d}t$.

For $E_2^2$ we use the estimate $\dst E(t)\leq \frac{5}{4}\left(\frac{\lambda^{1/2}}{p(t)}\right)^5$ that we established above.
It follows that
$$
E_2^2\leq \frac{1}{\sqrt{p(x)}}\int_x^y\frac{5\lambda^2}{4p^{9/2}(t)}
\frac{5}{4}\left(\frac{\lambda^{1/2}}{p(t)}\right)^5\,\mbox{d}t
\leq \frac{25\lambda^{9/2}}{16p(y)^{10}}|x-y|.
$$
If $|x|,|y|\leq\lambda/2$, we may use $|E(t)|\leq2\lambda^{-3}$, $q(t)\leq 5\lambda^{-2}$, $p(x)\geq \sqrt{3}\lambda/2$
to obtain $ E_2^2\leq\dst\frac{12}{\lambda^6} |x-y|$.

On the other hand,
$$
E_2^1\leq\frac{|h(0)|\sqrt{\lambda}}{\sqrt{p(x)}}\int_x^y\abs{\frac{q(t)}{p(t)}}\,\mbox{d}t
\leq\frac{5\times 2^{1/4}\lambda^{5/2}}{4\sqrt{\pi} n^{1/4}p(y)^{11/2}}|x-y|.
$$

If $n$ is odd, $h(0)=0$ while $|h'(0)|\leq \frac{2^{1/4}\lambda}{\sqrt{\pi}n^{1/4}}$ and we have to replace $\chi_1$
by
$$
\chi_1(x,t)=h'(0)\frac{q(t)}{\sqrt{\lambda}p(t)}\sin\ffi(x)\sin\bigl(\ffi(x)-\ffi(t)\bigr),
$$
from which we deduce that
$$
E_2^1\leq\frac{|h'(0)|}{\sqrt{\lambda}\sqrt{p(x)}}\int_x^y\abs{\frac{q(t)}{p(t)}}\,\mbox{d}t
\leq\frac{5\times 2^{1/4}\lambda^{5/2}}{4\sqrt{\pi} n^{1/4}p(y)^{11/2}}|x-y|\leq\frac{\lambda^{5/2}}{n^{1/4}p(y)^{11/2}}|x-y|.
$$
If $|x|,|y|\leq\lambda/2$, there is again a slight improvement:
$$
E_2^1\leq\frac{10\time 2^{7/4}}{\sqrt{\pi}n^{1/4}3^{3/4}\lambda^{13/2}}|x-y|\leq \frac{11}{\lambda^7}|x-y|,
$$
since $n\geq 3^{-1/4}\lambda^{1/2}$ if $n\geq 1$.
\smallskip

Finally,
\begin{eqnarray*}
\Phi(x,y)-\Phi(x,x)&=&\int_0^x\frac{q(t)}{\sqrt{p(t)}}h(t)\Bigl[\sin\bigl(\ffi(y)-\ffi(t)\bigr)
-\sin\bigl(\ffi(x)-\ffi(t)\bigr)\Bigr]\,\mbox{d}t\\
&=&2\int_0^x\frac{q(t)}{\sqrt{p(t)}}h(t)\cos\frac{\ffi(x)+\ffi(y)-2\ffi(t)}{2}\,\mbox{d}t
\sin\frac{\ffi(y)-\ffi(x)}{2}.
\end{eqnarray*}
The integral is estimated in the same way as we estimated $\Phi(x,x)$, while for $\ffi$
we use the mean value theorem and the fact that $\ffi'=p$. We, thus, get
$$
|E_3|\leq\frac{5\lambda^{5/2}p(x)}{4p(y)^5}|x-y|\leq \frac{5\lambda^{7/2}}{4p(y)^5}|x-y|.
$$
If $|x|,|y|\leq\lambda/2$, there is again a slight improvement: $|E_3|\leq \dst\frac{2}{\lambda^3}p(x)|x-y|\leq\frac{2}{\lambda^{5/2}}|x-y|$.

Summarizing,
%$$
%|E(x)-E(y)&|\leq& \ent{
%\frac{5\lambda^{7/2}}{8p(y)^{15/2}}+
%\frac{25\lambda^{9/2}}{16p(y)^{10}}+
%\frac{\lambda^{5/2}}{n^{1/4}p(y)^{11/2}}
%+\frac{5\lambda^{7/2}}{4p(y)^5}}|x-y|\\
%$$
\begin{eqnarray*}
|E(x)-E(y)|&\leq& \ent{\frac{p(y)^{5/2}}{2\lambda}+\frac{5}{4}+
\frac{4p(y)^{9/2}}{5 n^{1/4}\lambda^2 }+ \frac{p(y)^5}{\lambda}}\frac{5\lambda^{9/2}}{4p(y)^{10}}|x-y|\\
&\leq&\ent{\frac{p(y)^{3/2}}{2}+\frac{5}{4}+p(y)^2+ p(y)^4}\frac{5\lambda^{9/2}}{4p(y)^{10}}|x-y|,
\end{eqnarray*}
since $p(y)\leq\lambda$ and $n^{-1/4}\geq\frac{5}{4}p(y)^{1/2}$.
Now, assume that $|x|,|y|\leq(1-\eta)\lambda$, with $\lambda^{-2a}<\eta<1$, $a>0$. In particular,
$\lambda\geq p(y)\geq \lambda\sqrt{\eta}\geq\lambda^{1-a}$.
Thus,
$$
|E(x)-E(y)|\leq5\frac{\lambda^{4+9/2}}{\lambda^{10(1-a)}}|x-y|=\frac{5}{\lambda^{3/2-10a}}|x-y|.
$$
If $|x|,|y|\leq\lambda/2$, then
$$
|E(x)-E(y)|\leq\ent{\frac{1}{\lambda^{9/2}}+\frac{12}{\lambda^6}+\frac{11}{\lambda^7}+\frac{2}{\lambda^{5/2}}}
\leq \frac{8}{\lambda^{5/2}}|x-y|,
$$
since $\lambda\geq\sqrt{3}$.

\end{proof}

\subsection{Two technical lemmas}

We will now prove two technical lemmas. The first one concerns the function $\ffi_n$:

\begin{lemma}
\label{lem:simpasympt}
 If $|x|,|y|\leq T\leq\frac{1}{2}\sqrt{2n+1}$, then
\begin{equation}
\label{eq:lipffi0}
|\ffi_{n+1}(x)-\ffi_n(x)|\leq\frac{3T}{\sqrt{2n+1}},
\end{equation}
\begin{equation}
\label{eq:lipffi}
|\ffi_{n+1}(x)-\ffi_{n+1}(y)-\ffi_n(x)+\ffi_n(y)|\leq\frac{3}{\sqrt{2n+1}}|x-y|,
\end{equation}
\begin{equation}
\label{eq:lipffi2}
|\ffi_{n+1}(x)-\ffi_{n}(x)+\ffi_{n+1}(y)-\ffi_{n}(y)|\leq\frac{5T}{\sqrt{2n+1}},
\end{equation}
\begin{equation}
\label{eq:ffi}
\ffi_{n+1}(x)+\ffi_n(x)-\ffi_{n+1}(y)-\ffi_n(y)=(\sqrt{2n+1}+\sqrt{2n+3})(x-y)+\eps_n(x,y),
\end{equation}
with $\dst|\eps_n(x,y)|\leq\frac{T^2}{\sqrt{2n+1}}|x-y|$ and
\begin{equation}
\label{eq:lipffixxx}
|\ffi_{n}(x)-\ffi_{n}(y)|\leq\frac{5}{4}\sqrt{2n+1}|x-y|.
\end{equation}
\end{lemma}

\begin{proof} Note that \eqref{eq:lipffi0} is a direct consequence of \eqref{eq:lipffi} with $y=0$.

Recall that
$\dst\ffi_n(x)=\int_0^x\sqrt{2n+1-t^2}\,\mbox{d}t$. We have
\begin{eqnarray*}
|\ffi_{n+1}(x)-\ffi_n(x)-\ffi_{n+1}(y)+\ffi_n(y)|
&=&\abs{\int_y^x\sqrt{2n+3-t^2}-\sqrt{2n+1-t^2}\,\mbox{d}t}\\
&=&\abs{\int_y^x\frac{2}{\sqrt{2n+3-t^2}+\sqrt{2n+1-t^2}}\,\mbox{d}t}\\
&\leq&\abs{\int_y^x\frac{2}{\sqrt{2n+1-t^2}}\,\mbox{d}t}\\
&=&2\abs{\arcsin\frac{x}{\sqrt{2n+1}}-\arcsin\frac{y}{\sqrt{2n+1}}}.
\end{eqnarray*}
But, $\arcsin$ is $\frac{1}{\sqrt{\eta}}$-Lipschitz on $[-(1-\eta),(1-\eta)]$, thus,
$$
|\ffi_{n+1}(x)+\ffi_n(x)-\ffi_{n+1}(y)-\ffi_n(y)|\leq 2\sqrt{2}\frac{|x-y|}{\sqrt{2n+1}}.
$$

Next,
\begin{eqnarray*}
|\ffi_{n+1}(x)-\ffi_{n}(x)+\ffi_{n}(y)-\ffi_{n+1}(y)|
&\leq&\abs{\int_0^x\sqrt{2n+3-t^2}-\sqrt{2n+1-t^2}\,\mbox{d}t}\\
&&+\abs{\int_0^y\sqrt{2n+3-t^2}-\sqrt{2n+1-t^2}\,\mbox{d}t}\\
&\leq&2\int_0^T\frac{2}{\sqrt{2n+1-t^2}}\,\mbox{d}t\leq \frac{8}{\sqrt{3}}\frac{T}{\sqrt{2n+1}}.
\end{eqnarray*}

Set $N=\sqrt{2n+1}+\sqrt{2n+3}$. Then, $\ffi_{n+1}(x)+\ffi_n(x)-\ffi_{n+1}(y)-\ffi_n(y)$ is
\begin{eqnarray*}
&=&\int_y^x\sqrt{2n+3-t^2}+\sqrt{2n+1-t^2}\,\mbox{d}t\\
&=&N(x-y)+\int_y^x\sqrt{2n+3-t^2}+\sqrt{2n+1-t^2}-N\,\mbox{d}t.
\end{eqnarray*}
Therefore,
$$
\eps(x,y)=\int_y^x\sqrt{2n+3-t^2}-\sqrt{2n+3}\,\mbox{d}t+\int_y^x\sqrt{2n+1-t^2}-\sqrt{2n+1}\,\mbox{d}t.
$$
Let us estimate the second integral, the first being estimated in the same way:
\begin{eqnarray*}
\abs{\int_y^x\sqrt{2n+1-t^2}-\sqrt{2n+1}\,\mbox{d}t}&=&\abs{\int_y^x\frac{t^2}{\sqrt{2n+1-t^2}+\sqrt{2n+1}}\,\mbox{d}t}\\
&\leq&\frac{|x^3-y^3|}{3(1+\sqrt{3}/2)\sqrt{2n+1}}\leq\frac{T^2}{2\sqrt{2n+1}}|x-y|,
\end{eqnarray*}
since $\sqrt{2n+1-t^2}\geq \sqrt{3}/2$, when $|t|\leq T\leq \sqrt{2n+1}/2$.

Finally, \eqref{eq:ffi} implies \eqref{eq:lipffixxx}:
\begin{eqnarray*}
|\ffi_{n}(x)-\ffi_{n}(y)|&\leq &\sqrt{2n+1}|x-y|+|e_{n}(x)-e_{n}(y)|\leq\left(\sqrt{2n+1}+\frac{T^2}{2n+1}\right)|x-y|\\
&\leq&\frac{5}{4}\sqrt{2n+1}|x-y|,
\end{eqnarray*}
since $T\leq\sqrt{2n+1}/2$.
\end{proof}

\begin{remark} Geometrically, $|\ffi_{n+1}(x)-\ffi_n(x)-\ffi_{n+1}(y)+\ffi_n(y)|$
is the area of the intersection of the annulus of inner radius $\sqrt{2n+1}$ an outer radius $\sqrt{2n+3}$
with a vertical strip with first coordinate in $[x,y]$. The annulus has width $o(n^{-1/2})$
so that its intersection with the strip has area $o(n^{-1/2}|x-y|)$ as long as this strip
is not ``tangent'' to the annulus. The lemma is a quantitative statement of this simple
geometric fact.
\end{remark}

The next result is a simplification of Theorem \ref{th:approxherm}:

\begin{corollary}
\label{cor:simpasympt}
Let $T\geq 2$ and let $n\geq 2T^2$. Then, for $|x|\leq T$, we obtain that

-- if $n$ is even, $n=2p$
\begin{equation}
\label{eq:approxhermpair}
h_{2p}(x)=\frac{(-1)^p}{\sqrt{\pi}p^{1/4}}\cos\ffi_{2p}(x)+\tilde E_{2p}(x);
\end{equation}

-- if $n$ is odd, $n=2p+1$
\begin{equation}
\label{eq:approxhermimpair}
h_{2p+1}(x)=
\frac{(-1)^p}{\sqrt{\pi}p^{1/4}}\sin\ffi_{2p+1}(x)+\tilde E_{2p+1}(x),
\end{equation}
where, for $|x|,|y|\leq T$,
\begin{equation}
\label{eq:esttildeE}
|\tilde E_n(x)|\leq\frac{2T^2}{(2n+1)^{5/4}}
\quad\mbox{and}\quad
|\tilde E_n(x)-\tilde E_n(y)|\leq 3\frac{T^2}{(2n+1)^{3/4}}|x-y|
\end{equation}
\end{corollary}

\begin{proof}
First, we consider the case when $n$ is even, $n=2p$. Then, $h_{2p}(0)=\frac{(-1)^p}{\sqrt{\pi}p^{1/4}}\left(1-\frac{\eta_{2p}}{8p}\right)$ and $h_{2p+1}^\prime(0)=0$. Therefore,
\eqref{eq:approxherm} reads
\begin{eqnarray*}
h_{2p}(x)&=&\frac{(-1)^p}{\sqrt{\pi}p^{1/4}}\left(1-\frac{\eta_{2p}}{8p}\right)\left(\frac{4p+1}{4p+1-x^2}\right)^{1/4}\cos\ffi_{2p}(x)
+\tilde E_{2p}(x)\\
&=&\frac{(-1)^p}{\sqrt{\pi}p^{1/4}}\cos\ffi_{2p}(x)+\tilde E_{2p}(x),
\end{eqnarray*}
where $\tilde E_{2p}(x)$ is
\begin{eqnarray*}
&=&\frac{(-1)^p}{\sqrt{\pi}p^{1/4}}\left(1-\frac{\eta_{2p}}{8p}\right)\ent{\left(\frac{4p+1}{4p+1-x^2}\right)^{1/4}-1}
\cos\ffi_{2p}(x)+\frac{(-1)^p}{\sqrt{\pi}}\frac{\eta_{2p}}{8p^{5/4}}\cos\ffi_{2p}(x)\\
&&+E_{2p}(x).
\end{eqnarray*}
But, $(1+a)^{1/4}-1\leq\frac{a}{4}$, which
for $0\leq a:=\frac{x^2}{4p+1-x^2}\leq \frac{T^2}{4p+1-T^2}\leq\frac{4}{3}\frac{T^2}{4p+1}$
gives
\begin{equation}
\label{eq:aaaa}
\abs{\left(\frac{4p+1}{4p+1-x^2}\right)^{1/4}-1}\leq\frac{1}{3}\frac{T^2}{4p+1}.
\end{equation}
It follows that
\begin{eqnarray*}
|\tilde E_{2p}(x)|&\leq& \frac{1}{\sqrt{\pi}p^{1/4}}\abs{\left(1+\frac{x^2}{4p+1-x^2}\right)^{1/4}-1}
+\frac{1}{8\sqrt{\pi}p^{5/4}}+|E_{2p}(x)|\\
&\leq&\frac{1}{\sqrt{\pi}p^{1/4}}\frac{1}{3}\frac{T^2}{4p+1}+\frac{1}{8\sqrt{\pi}p^{5/4}}+\frac{2}{(4p+1)^{3/2}}
\leq\frac{2T^2}{(4p+1)^{5/4}}.
\end{eqnarray*}

Further,
\begin{eqnarray*}
|\tilde E_{2p}(x)-\tilde E_{2p}(y)|&\leq&
\frac{1}{\sqrt{\pi}p^{1/4}}\abs{\left(1-\frac{x^2}{4p+1}\right)^{-1/4}-\left(1-\frac{y^2}{4p+1}\right)^{-1/4}}\\
&&+\frac{1}{\sqrt{\pi}p^{1/4}}\ent{\abs{\left(\frac{4p+1}{4p+1-y^2}\right)^{1/4}-1}+\frac{1}{8p}}\abs{\cos\ffi_{2p}(x)-\cos\ffi_{2p}(y)}\\
&&+|E_{2p}(x)-E_{2p}(y)|=E_{2p}^1(x,y)+E_{2p}^2(x,y)+E_{2p}^3(x,y).
\end{eqnarray*}
We have already established that $E_{2p}^3(x,y)\leq\dst\frac{8}{(4p+1)^{5/4}}|x-y|$.
Further, if $0\leq X,Y\leq \frac{T^2}{4p+1}\leq\frac{1}{4}$, then
$$
|(1-X)^{-1/4}-(1-Y)^{-1/4}|\leq\frac{5}{4}|X-Y|
\sup_{0\leq t\leq 1/4}(1-t)^{-5/4}=\frac{5\sqrt{2}}{3^{1/4}}|X-Y|.
$$
Therefore
$$
E_{2p}^1(x,y)\leq \frac{1}{\sqrt{\pi}p^{1/4}}\frac{5\sqrt{2}}{3^{1/4}}\frac{|x^2-y^2|}{4p+1}
\leq 4\frac{T}{4p+1}|x-y|.
$$
Finally, for $E_{2p}^2$ we use the fact that $\cos$ is $1$-Lipschitz, \eqref{eq:lipffixxx}
and \eqref{eq:aaaa}, to obtain
$$
E_{2p}^2(x,y)\leq \frac{1}{\sqrt{\pi}p^{1/4}}\ent{\frac{1}{3}\frac{T^2}{4p+1}+\frac{1}{8p}}\frac{5}{4}\sqrt{4p+1}|x-y|
%\leq\frac{T^2}{(4p+1)^{3/4}}\frac{5(4+\eps)^{1/4}}{8\sqrt{\pi}}|x-y|
\leq 2\frac{T^2}{(4p+1)^{3/4}}|x-y|.
$$
Thus,
$$
|\tilde E_{2p}(x)-\tilde E_{2p}(y)|\leq \left(4\frac{T}{4p+1}+2\frac{T^2}{(4p+1)^{3/4}}+\frac{8}{(4p+1)^{5/4}}\right)|x-y|
\leq 3\frac{T^2}{(4p+1)^{3/4}}|x-y|.
$$

\medskip

Let us now consider the case when $n$ is odd, $n=2p+1$. Then, $h_{2p+1}(0)=0$ and
$h_{2p+1}^\prime(0)=\frac{(-1)^p\sqrt{4p+3}}{\sqrt{\pi}p^{1/4}}\left(1-\frac{\eta_{2p+1}}{4}\right)$.
Therefore \eqref{eq:approxherm} reads
\begin{eqnarray*}
h_{2p+1}(x)&=&\frac{(-1)^p\sqrt{4p+3}}{\sqrt{\pi}p^{1/4}}\left(1-\frac{\eta_{2p+1}}{4}\right)
\frac{\sin\ffi_{2p+1}(x)}{\bigl((4p+3)(4p+3-x^2)\bigr)^{1/4}}+E_{2p+1}(x)\\
&=&\frac{(-1)^p}{\sqrt{\pi}p^{1/4}}\left(1-\frac{\eta_{2p+1}}{4}\right)
\left(\frac{4p+3}{(4p+3-x^2)}\right)^{1/4}\sin\ffi_{2p+1}(x)+E_{2p+1}(x)\\
&=&\frac{(-1)^p}{\sqrt{\pi}p^{1/4}}\sin\ffi_{2p+1}(x)+\tilde E_{2p+1}(x),
\end{eqnarray*}
where
\begin{eqnarray*}
\tilde E_{2p+1}(x)&=&\frac{(-1)^p}{\sqrt{\pi}p^{1/4}}\left(1-\frac{\eta_{2p+1}}{4}\right)\ent{\left(\frac{4p+3}{(4p+3-x^2)}\right)^{1/4}-1}\sin\ffi_{2p+1}(x)\\
&&+\frac{(-1)^p\eta_{2p+1}}{4\sqrt{\pi}p^{1/4}}\sin\ffi_{2p+1}(x)+E_{2p+1}(x).
\end{eqnarray*}
The remaining of the proof is the same as for $\tilde E_{2p}$.
\end{proof}

\begin{remark}
The assumption $T\geq 2$ is here to make it easier to group terms in the estimates of the errors.
For $T\geq 1$ the constants are slightly worse. The reader may check that
\begin{equation}
\label{eq:esttildeET1}
|\tilde E_n(x)|\leq\frac{3T^2}{(2n+1)^{5/4}}
\quad\mbox{and}\quad
|\tilde E_n(x)-\tilde E_n(y)|\leq 8\frac{T^2}{(2n+1)^{3/4}}|x-y|.
\end{equation}
\end{remark}

\section{$L^2$-Approximation of functions by Hermite functions}%\label{sec:main}

\subsection{The kernel of the projection onto the Hermite functions}\label{sec:kernel}

As $(h_n)_{n\geq 0}$ forms an orthonormal basis of $L^2(\R)$, every $f\in L^2(\R)$ can be written as
$$
f(x)=\lim_{n\to+\infty}\sum_{k=0}^n \scal{f,h_k}h_k(x),
$$
where the limit is in the $L^2(\R)$ sense. Further,
$$
\sum_{k=0}^n \scal{f,h_k}h_k(x)= \sum_{k=0}^n\int_{\R}f(y)h_k(y)\,\mbox{d}y\,h_k(x)
=\int_{\R}f(y)\sum_{k=0}^nh_k(x)h_k(y)\,\mbox{d}y=\int_{\R} k_n(x,y)f(y)\,\mbox{d}y,
$$
with the kernel
$\dst
k_n(x,y)=\sum_{k=0}^nh_k(x)h_k(y).
$
According to the Christoffel-Darboux Formula,
$$
k_n(x,y)=\sqrt{\frac{n+1}{2}}\frac{h_{n+1}(x)h_n(y)-h_{n+1}(y)h_n(x)}{x-y}.
$$
We will now use Corollary \ref{cor:simpasympt} to approximate this kernel:

\begin{theorem}
\label{th:estkn}
Let $T\geq 2$, $n\geq 2T^2$ and $N=\frac{\sqrt{2n+1}+\sqrt{2n+3}}{2}$.
Then, for  $|x|,|y|\leq T$,
$$
k_n(x,y)=\frac{1}{\pi}\frac{\sin N(x-y)}{x-y}+R_n(x,y),
$$
with $|R_n(x,y)|\leq\dst\frac{17T^2}{\sqrt{2n+1}}$.
\end{theorem}

\begin{remark} The same estimate holds for $T=1$ provided $n\geq 6$.
\end{remark}

\begin{proof} For sake of simplicity, we will only prove the theorem in the case when $n$ is even and write $n=2p$.

Let $\lambda=\sqrt{2n+1}$, $\mu=\sqrt{2n+3}$, $\alpha=\frac{1}{\sqrt{\pi}p^{1/4}}$,
$\beta=\frac{1}{\sqrt{\pi}p^{1/4}}$, $E=(-1)^p\tilde E_{2p}$ and $F=(-1)^p\tilde E_{2p+1}$.
Then, according to Lemma \ref{lem:simpasympt},
$$
\left\{\begin{matrix}h_{2p}(x)&=&(-1)^p\bigl(\frac{1}{\sqrt{\pi}p^{1/4}}\cos\ffi_{2p}(x)+E(x)\bigr)\\
h_{2p+1}(x)&=&(-1)^p\bigl(\frac{1}{\sqrt{\pi}p^{1/4}}\sin\ffi_{2p+1}(x)+F(x)\bigr)
\end{matrix}\right..
$$
Therefore, $h_{2p+1}(x)h_{2p}(y)-h_{2p+1}(y)h_{2p}(x)$ is
\begin{eqnarray*}
&=&
\frac{1}{\pi p^{1/2}}\bigl(\sin\ffi_{2p+1}(x)\cos\ffi_{2p}(y)-\sin\ffi_{2p+1}(y)\cos\ffi_{2p}(x)\bigr)\\
&&+\frac{1}{\sqrt{\pi}p^{1/4}}\bigl(F(x)\cos\ffi_{2p}(y)-F(y)\cos\ffi_{2p}(x)\bigr)\\
&&+\frac{1}{\sqrt{\pi}p^{1/4}}\bigl(\sin\ffi_{2p+1}(x)E(y)-\sin\ffi_{2p+1}(y)E(x)\bigr)\\
&&+F(x)E(y)-F(y)E(x)\\
&=&H_1(x,y)+H_2(x,y)+H_3(x,y)+H_4(x,y).
\end{eqnarray*}

--- The first term in the equation above is the principal one. Let us start by computing
\begin{eqnarray*}
A:=\sin\ffi_{2p+1}(x)\cos\ffi_{2p}(y)&-&\sin\ffi_{2p+1}(y)\cos\ffi_{2p}(x)\\
&=&\frac{1}{2}\bigl[\sin\bigl(\ffi_{2p+1}(x)+\ffi_{2p}(y)\bigr)-\sin\bigl(\ffi_{2p+1}(x)-\ffi_{2p}(y)\bigr)\\
&&\qquad-\sin\bigl(\ffi_{2p+1}(y)+\ffi_{2p}(x)\bigr)+\sin\bigl(\ffi_{2p+1}(y)-\ffi_{2p}(x)\bigr)\bigr]\\
&=&\sin\frac{\ffi_{2p+1}(x)-\ffi_{2p+1}(y)-\ffi_{2p}(x)+\ffi_{2p}(y)}{2}\\
&&\times\cos\frac{\ffi_{2p+1}(x)+\ffi_{2p+1}(y)+\ffi_{2p}(x)+\ffi_{2p}(y)}{2}\\
&&+\sin\frac{\ffi_{2p+1}(y)+\ffi_{2p}(y)-\ffi_{2p}(x)-\ffi_{2p+1}(x)}{2}\\
&&\times\cos\frac{\ffi_{2p+1}(x)-\ffi_{2p}(x)-\ffi_{2p}(y)+\ffi_{2p+1}(y)}{2}\\
&=&S_1C_1+S_2C_2.
\end{eqnarray*}
Now, according to \eqref{eq:lipffi},
$$
|S_1C_1|\leq|S_1|\leq\frac{|\ffi_{2p+1}(x)-\ffi_{2p+1}(y)-\ffi_{2p}(x)+\ffi_{2p}(y)|}{2}\leq\frac{3}{2\sqrt{2n+1}}|x-y|,
$$
while $S_2C_2= S_2(1+C_2-1)$. But, with \eqref{eq:lipffi2},
$$
|C_2-1|\leq\frac{|\ffi_{2p+1}(x)-\ffi_{2p}(x)-\ffi_{2p}(y)+\ffi_{2p+1}(y)|^2}{2}\leq\frac{25T^2}{2(2n+1)}.
$$
Thus, with \eqref{eq:ffi},
$$
|S_2(C_2-1)|\leq\left(N+\frac{T^2}{\sqrt{2n+1}}\right)|x-y|\frac{25T^2}{2(2n+1)}\leq\frac{16T^2}{\sqrt{2n+1}}|x-y|.
$$
Finally, using again Lemma \ref{lem:simpasympt}, $\sin\bigl(N(y-x)+\eps_n(y,x)\bigr)$ is
\begin{eqnarray*}
&=&\sin N(y-x)+\sin N(y-x)\bigl(\cos\eps_n(y,x)-1\bigr)+\cos N(y-x)\sin\eps_n(x,y)\\
&=&\sin N(y-x)+E_2(x,y),
\end{eqnarray*}
where
$$
|E_2(x,y)|\leq |\eps_n(x,y)|+\frac{|\eps_n(x,y)|^2}{2}\leq \frac{2T^2}{\sqrt{2n+1}}|x-y|.
$$
Grouping those estimates leads to
$$
A=\sin N(y-x)+E_3(x,y)\quad\mbox{with}\quad |E_3(x,y)|\leq\frac{39T^2}{2\sqrt{2n+1}}|x-y|
$$
Notice, that
$$
\frac{1}{\pi p^{1/2}}=\frac{1}{\pi}\sqrt{\frac{2}{n+1}}\sqrt{1+\frac{1}{n}}=\sqrt{\frac{2}{n+1}}\frac{1}{\pi}\left(1
+\frac{\xi_n}{n}\right)
$$
with $|\xi_n|\leq1/2$.

We, thus, conclude that $H_1(x,y)=\dst\sqrt{\frac{2}{n+1}}\left(\frac{1}{\pi}\sin N(y-x)+E_4(x,y)\right)$, with
$$
|E_4(x,y)|\leq\dst\frac{1}{\pi}\left(1+\frac{\xi_n}{n}\right)|E_3(x,y)|+\frac{\xi_n}{\pi n}N|x-y|
\leq \frac{5T^2}{\sqrt{2n+1}}|x-y|.
$$
\medskip

--- Consider
$$
F(x)\cos\ffi_{2p}(y)-F(y)\cos\ffi_{2p}(x)=F(x)\bigl(\cos\ffi_{2p}(y)-\cos\ffi_{2p}(x)\bigr)
+\bigl(F(x)-F(y)\bigr)\cos\ffi_{2p}(x).
$$
Then, according to \eqref{eq:esttildeE},
$\abs{\bigl(F(x)-F(y)\bigr)\cos\ffi_{2p}(x)}\leq|F(x)-F(y)|\leq\frac{3T^2}{(2n+1)^{3/4}}|x-y|$, while
\begin{eqnarray*}
\abs{F(x)\bigl(\cos\ffi_{2p}(y)-\cos\ffi_{2p}(x)\bigr)}&\leq&\frac{2T^2}{(2n+1)^{5/4}}|\ffi_{2p}(y)-\ffi_{2p}(x)|\\
&\leq&\frac{2T^2}{(2n+1)^{5/4}}\frac{5}{4}\sqrt{2n+1}|x-y|=\frac{5T^2}{2(2n+1)^{3/4}}|x-y|,
\end{eqnarray*}
with \eqref{eq:lipffixxx}. Therefore,
$$
|H_2(x,y)|\leq\frac{1}{\sqrt{\pi}p^{1/4}}\frac{(5/2+3)T^2}{(2n+1)^{3/4}}|x-y|
\leq\sqrt{\frac{2}{n+1}}\frac{3T^2}{2(2n+1)^{1/2}}|x-y|.
$$
Similarly, the estimate $|H_3(x,y)|\leq\sqrt{\frac{2}{n+1}}\frac{3T^2}{2(2n+1)^{1/2}}|x-y|$
holds.

\smallskip

Note that for $T=1$, we have to use \eqref{eq:esttildeET1} instead of \eqref{eq:esttildeE}
which gives 
$$
|H_2(x,y)|,|H_3(x,y)|\leq\sqrt{\frac{2}{n+1}}\frac{5T^2}{(2n+1)^{1/2}}|x-y|.
$$

\medskip

--- Finally, according to \eqref{eq:esttildeE},
\begin{eqnarray*}
|F(x)E(y)-F(y)E(x)|&\leq&|F(x)||E(y)-E(x)|+|E(x)||F(x)-F(y)|\\
&\leq&\frac{12T^4}{(2n+1)^2}|x-y|\leq \sqrt{\frac{2}{n+1}}\frac{2T^2}{\sqrt{2n+1}}|x-y|.
\end{eqnarray*}

\smallskip

Note that for $T=1$, we have to use \eqref{eq:esttildeET1} instead of \eqref{eq:esttildeE}
which gives $|H_4(x,y)|\leq\sqrt{\frac{2}{n+1}}\frac{24}{(2n+1)^{3/2}}|x-y|$

\medskip

Grouping terms together, we obtain,
$$
h_{2p+1}(x)h_{2p}(y)-h_{2p+1}(y)h_{2p}(x)=\sqrt{\frac{2}{n+1}}\left(\frac{1}{\pi}\sin N(y-x)+E_5(x,y)\right),
$$
with $\dst |E_5(x,y)|\leq \frac{9T^2}{\sqrt{2n+1}}|x-y|$.
\end{proof}

\subsection{A tail estimate}\label{sec:kernel2}
Let us now establish a tail estimate for $k_n$.

\begin{proposition}
\label{prop:tail}
Let $T\geq 2$ and $n\geq 2T^2$. Then,
for $|x|\leq T$,
$$
\int_{|y|\geq 2T} k_n(x,y)^2\,\mbox{d}y\leq \frac{2}{\pi^2 T}+\frac{12T^2}{\sqrt{2n+1}}\ln(2n+1).
$$
\end{proposition}

\begin{proof}
First, using the reproducing kernel property of $k_n$,
$$
\int_\R k_n(x,y)k_n(z,y)\,\mbox{d}y=k_n(x,z).
$$
But, since $\dst k_n(x,y)=\sum_{k=0}^nh_k(x)h_k(y)$ and  $h_k=H_ke^{-x^2/2}$, with $H_k$ a polynomial of degree $k$,
there exists a constant $C_n$, such that
$$
|k_n(x,y)|\leq C_n(1+|x|)^n(1+|y|)^ne^{-(x^2+y^2)/2}.
$$
Applying Lebesgue's Dominated Converence Theorem, we have
$$
\dst\int_\R k_n(x,y)k_n(z,y)\,\mbox{d}y\to \int_\R k_n(x,y)^2\,\mbox{d}y,
$$
when $z\to x$. On the other hand, Theorem \ref{th:estkn} shows that
$$
k_n(x,z)\to\frac{N}{\pi}+R_n(x,x)
$$
uniformly in $x\in[-T,T]$. Therefore,
\begin{equation}
\label{eq:estnormkn}
\int_\R k_n(x,y)^2\,\mbox{d}y=\frac{N}{\pi}+R_n(x,x),\quad |R_n(x)|\leq\frac{9T^2}{\sqrt{2n+1}}.
\end{equation}

Now, for $|x|\leq T$,  Theorem \ref{th:estkn} shows that
\begin{eqnarray*}
\int_{[-2T,2T]} k_n(x,y)^2\,\mbox{d}y&=&\int_{[-2T,2T]}\left(\frac{1}{\pi}\frac{\sin N(y-x)}{y-x}
+R_n(x,y)\right)^2\,\mbox{d}y\\
&=&\frac{1}{\pi^2}\int_{[-2T,2T]}\frac{\sin^2 N(y-x)}{(y-x)^2}\,\mbox{d}y+R_n(x).
\end{eqnarray*}
The estimation of the first term is classical: for $|x|\leq T$,
\begin{eqnarray*}
\frac{1}{\pi^2}\int_{-2T}^{2T}\frac{\sin^2 N(y-x)}{(y-x)^2}\,\mbox{d}y
&=&\frac{N}{\pi^2}\int_{-N(2T+x)}^{N(2T-x)}\frac{\sin^2 z}{z^2}\,\mbox{d}z\\
&=&\frac{N}{\pi^2}\int_\R\frac{\sin^2 z}{z^2}\,\mbox{d}z
-\frac{N}{\pi^2}\left(\int_{-\infty}^{-N(2T+x)}+\int_{N(2T-x)}^{+\infty}\right)\frac{\sin^2 z}{z^2}\,\mbox{d}z\\
&=&\frac{N}{\pi}-R^1_N(x),
\end{eqnarray*}
where $\dst 0\leq R^1_N(x)\leq \frac{2N}{\pi^2}\int_{N T}^{+\infty}\frac{\mbox{d}z}{z^2}= \frac{2}{\pi^2 T}$.

Next, we write $R_n(x)=R_n^2(x)+R_n^3(x)$, where
$\dst
R_n^2(x)=\int_{[-2T,2T]}R_n(x,y)^2\,\mbox{d}y\geq 0%\leq 4\times 3^4\frac{T^5}{2n+1}
$
and
\begin{eqnarray*}
|R_n^3|&=&\int_{[-T,T]}\frac{2}{\pi}\abs{\frac{\sin N(y-x)}{y-x}R_n(x,y)}\,\mbox{d}y\\
&\leq&\frac{18T^2}{\pi\sqrt{2n+1}}\int_{-N(T-x)}^{N(T-x)}\abs{\frac{\sin z}{z}}\,\mbox{d}z\\
&\leq&\frac{36T^2}{\pi\sqrt{2n+1}}\int_0^{2NT}\min(1,z^{-1})\,\mbox{d}z
\leq \frac{12T^2}{\sqrt{2n+1}}\ln(2n+1).
\end{eqnarray*}

It follows that for $|x|\leq T$
\begin{eqnarray*}
\int_{|y|\geq 2T} k_n(x,y)^2\,\mbox{d}y&\leq&R_n^1+|R_n^3|\leq \frac{2}{\pi^2 T}+\frac{12T^2}{\sqrt{2n+1}}\ln(2n+1)
\end{eqnarray*}
as announced.
\end{proof}

\subsection{Approximating almost time and band limited functions by Hermite functions}\label{sec:mainth}

We can now prove Theorem \ref{th:mainhermite}:

\begin{theorem}\label{th:projhermite}
Let $\Omega_0,T_0\geq2$ and $\eps_T,\eps_\Omega>0$. Assume that
$$
\int_{|t|>T_0}|f(t)|^2\,\mathrm{d}t\leq \eps_T^2\norm{f}_{L^2(\R)}^2
\quad\mbox{and}\quad
\int_{|\omega|>\Omega_0}|\widehat{f}(\omega)|^2\,\mathrm{d}\omega\leq \eps_\Omega^2\norm{f}_{L^2(\R)}^2.
$$
For $n$ an integer, let $K_nf$ be the orthogonal projection of $f$
on the span of $h_0,\ldots,h_n$.

Assume that $n\geq \max(2T^2,2\Omega^2)$. Then, for $T\geq T_0$,
\begin{equation}
\label{eq:approxhermloc}
\norm{f-K_nf}_{L^2([-T,T])}
\leq\left(2\eps_T+\eps_\Omega+\frac{34T^3}{\sqrt{2n+1}}\right)\norm{f}_{L^2(\R)}
\end{equation}
and, for $T\geq 2T_0$,
\begin{equation}
\label{eq:approxhermglob}
\norm{f-K_nf}_{L^2(\R\setminus[-T,T])}
\left(2\eps_T+\frac{1}{2 T^{1/2}}+\frac{12T^{5/2}}{\sqrt{2n+1}}\ln(2n+1)\right)^{1/2}\norm{f}_{L^2(\R)}.
\end{equation}
\end{theorem}

\begin{remark}
As the proof of \eqref{eq:approxhermloc} only depends on Theorem \ref{th:estkn}, this estimate
holds for $T=1$, provided we assume that $n\geq 6$ (see the remark following Theorem \ref{th:estkn}).
\end{remark}

\begin{proof}
We will introduce several projections. For $T,\Omega>0$, let
$$
P_Tf=\mathbf{1}_{[-T,T]}f\quad\mbox{and}\quad Q_\Omega f=\ff^{-1}\bigl[\mathbf{1}_{[-\Omega,\Omega]}\widehat{f}].
$$
A simple computation shows that
$$
Q_\Omega f(x)=\frac{1}{\pi}\int_{\R}\frac{\sin\Omega(x-y)}{x-y}f(y)\,\mbox{d}y.
$$
The hypothesis on $f$ is that $\norm{f-P_Tf}_{L^2(\R)}\leq \eps_T\norm{f}_{L^2(\R)}$ for $T\geq T_0$
and $\norm{f-Q_\Omega f}_{L^2(\R)}\leq\eps_\Omega\norm{f}_{L^2(\R)}$ for $\Omega\geq\Omega_0$.

Finally, recall that the projection on the $n$ first Hermite functions, is given by
$$
K_nf(x)=\sum_{k=0}^n\scal{f,h_k}h_k(x)=\int_\R k_n(x,y)f(y)\,\mbox{d}y.
$$

\medskip

It is enough to prove \eqref{eq:approxhermloc} for $T=T_0$.
Let us recall the integral operator
$$
\mathcal{R}_n^T f(x)=\int_{[-T,T]}R_n(x,y)f(y)\,\mbox{d}y,
$$
where $R_n(x,y)$ are defined in Theorem \ref{th:estkn}.
Notice that $k_n(x,y)=k_n(y,x)$ so that $R_n(x,y)=R_n(y,x)$.
We may then reformulate Theorem \ref{th:estkn} as following:
$$
P_TK_nP_Tf=P_TQ_NP_Tf+P_T\mathcal{R}_n^T P_Tf,
$$
where $N=\frac{\sqrt{2n+1}+\sqrt{2n+3}}{2}$. Note that $N\geq\Omega_0$.
By using \eqref{th:estkn}, it is easy to see that 
%\rea{Use of the Hilbert-Schmidt norm instead of Schur's lemma}
\begin{eqnarray}
\norm{P_T \mathcal R_n^T P_Tf}_{L^2(\R)}&\leq& \norm{P_T \mathcal R_n^T P_T}_{L^2(\R)\to L^2(\R)}\norm{f}_{L^2(\R)}
\|  \mathcal R_n^T\|_{HS} \norm{f}_{L^2(\R)}\nonumber\\
&\leq&\frac{34T^3}{\sqrt{2n+1}}\norm{f}_{L^2(\R)}.\label{norm_Rn}
\end{eqnarray}
%
%Notice that, using Minkowski's inequality
%\begin{eqnarray*}
%\norm{P_TT_nP_Tf}_{L^2(\R)}&=&\left(\int_{-T}^T\abs{\int_{-T}^T\eps_{2,n}(x,y)f(y)\,\mbox{d}y}^2\,\mbox{d}x\right)^{1/2}\\
%&\leq&\int_{-T}^T\left(\int_{-T}^T|\eps_{2,n}(x,y)|^2\,\mbox{d}x\right)^{1/2}|f(y)|\,\mbox{d}y\\
%&\leq&16\sqrt{2}\frac{T^4}{n^{1/12}}\int_{-T}^T|f(y)|\,\mbox{d}y\leq 32\frac{T^{9/2}}{n^{1/12}}\norm{f}_{L^2(\R)}.
%\end{eqnarray*}

Now, using the fact that projections are contractive and $N\geq\Omega_0$, we have
\begin{eqnarray*}
\norm{f-K_nf}_{L^2([-T,T])}&=&\norm{P_Tf-P_TK_nf}_{L^2(\R)}\\
&\leq& \norm{P_Tf-P_TK_nP_Tf}_{L^2(\R)}+\norm{P_TK_n(f-P_Tf)}_{L^2(\R)}\\
&\leq&\norm{P_Tf-P_TQ_NP_Tf+P_T\mathcal{R}_n^T P_Tf}_{L^2(\R)}+\norm{f-P_Tf}_{L^2(\R)}\\
&\leq&\norm{P_Tf-P_TQ_NP_Tf}_{L^2(\R)}+\norm{P_T\mathcal{R}_n^T P_Tf}_{L^2(\R)}+\norm{f-P_Tf}_{L^2(\R)}.
\end{eqnarray*}
Now, write $P_TQ_NP_Tf=P_TQ_Nf+P_TQ_N(f-P_Tf)$, then
\begin{eqnarray*}
\norm{P_Tf-P_TQ_NP_Tf}_{L^2(\R)}&\leq& \norm{P_Tf-P_TQ_Nf}_{L^2(\R)}+\norm{P_TQ_N(f-P_Tf)}_{L^2(\R)}\\
&\leq&\norm{f-Q_Nf}_{L^2(\R)}+\norm{f-P_Tf}_{L^2(\R)}.
\end{eqnarray*}

Therefore,
\begin{eqnarray*}
\norm{f-K_nf}_{L^2([-T,T])}
&\leq&\norm{f-Q_Nf}_{L^2(\R)}+\frac{34T^3}{\sqrt{2n+1}}\norm{f}_{L^2(\R)}+2\norm{f-P_Tf}_{L^2(\R)}\\
&\leq&\left(\eps_\Omega+\frac{34T^3}{\sqrt{2n+1}}+2\eps_T\right)\norm{f}_{L^2(\R)},
\end{eqnarray*}
since $N\geq\Omega_0$.

\medskip

Let us now prove \eqref{eq:approxhermglob}. It is enough to prove it for $T=2T_0$.
Note that
\begin{eqnarray*}
\norm{f-K_nf}_{L^2(\R\setminus[-2T_0,2T_0])}&\leq& \norm{f}_{L^2(\R\setminus[-2T_0,2T_0])}+\norm{K_nP_Tf}_{L^2(\R\setminus[-2T_0,2T_0])}\\
&&+\norm{K_n(f-P_{T_0})}_{L^2(\R)}\\
&\leq& 2\eps_T\norm{f}_{L^2(\R)}^2+\norm{K_nP_{T_0}f}_{L^2(\R\setminus[-2T_0,2T_0])}.
\end{eqnarray*}
We, therefore, need to estimate
$$
\norm{K_nP_Tf}_{L^2(\R\setminus[-2T_0,2T_0])}=\left(\int_{|x|\geq 2T_0}\abs{\int_{|y|\leq T_0}k_n(x,y)f(y)\,\mbox{d}y}^2\,\mbox{d}x\right)^{1/2}.
$$
Using Minkowski's inequality, this quantity is bounded by
$$
\int_{|y|\leq T_0}\left(\int_{|x|\geq 2T_0}\abs{k_n(x,y)f(y)}^2\,\mbox{d}x\right)^{1/2}\,\mbox{d}y
=\int_{|y|\leq T_0}\left(\int_{|x|\geq 2T_0}\abs{k_n(x,y)}^2\,\mbox{d}x\right)^{1/2}|f(y)|\,\mbox{d}y
$$
\begin{eqnarray*}
&\leq&\left(\sup_{|y|\leq T_0}\int_{|x|\geq 2T_0}\abs{k_n(x,y)}^2\,\mbox{d}x\right)^{1/2}
\left(\int_{|y|\leq T_0}|f(y)|\,\mbox{d}y\right)^{1/2}\\
&\leq&2\left(\frac{2}{\pi^2 T_0}+\frac{6T_0^2}{\sqrt{2n+1}}\ln(2n+1)\right)^{1/2}\norm{f}_{L^1([-T_0,T_0])}\\
&\leq&2\sqrt{2}\left(\frac{1}{\pi^2 T_0^{1/2}}+\frac{6T_0^{5/2}}{\sqrt{2n+1}}\ln(2n+1)\right)^{1/2}\norm{f}_{L^2(\R)},
\end{eqnarray*}
which is, with Proposition \ref{prop:tail}, complete the proof.
\end{proof}

\rep{Extended remark}
\begin{remark}
The error estimate given by \eqref{eq:approxhermloc} is not practical due to the  low decay rate
of the bound of $\|\mathcal R_n^T\|$ given by ${\displaystyle \frac{34 T^3}{\sqrt{2n+1}}}.$  
By replacing this later with  a non explicit but a more realistic error estimate $\|\mathcal R_n^T\|_{HS},$
one gets the following  error estimate  which is more practical for numerical purposes,
 \begin{equation}\label{eq2:approxhermloc}
  \norm{f-K_nf}_{L^2([-T,T])}
 \leq\left(\eps_\Omega+\| \mathcal R_n^T\|_{HS}+2\eps_T\right)\norm{f}_{L^2(\R)}.
 \end{equation}
Note also that the factor of $\| \mathcal R_n^T\|_{HS}$ is actually $\norm{f}_{L^2([-T,T])}$, to see this, it is enough
to write $P_Tf=P_TP_Tf$ in \eqref{norm_Rn}. If one has an $L^1$ bound for $f$, one may replace this term with the following
computation:
\begin{eqnarray}
\norm{P_T \mathcal R_n^T P_Tf}_{L^2(\R)}^2&=&\int_{-T}^T\abs{\int_{-T,T}R_n(x,y)f(y)\,\mbox{d}y}^2\,\mbox{d}x\\
&\leq&\int_{-T}^T\sup_{y\in[-T,T]}|R_n(x,y)|^2\,\mbox{d}x\,\left(\int_{-T}^T|f(y)|\,\mbox{d}y\right)^2.
\end{eqnarray}
Thus, with Theorem \ref{th:estkn}, one obtains
$$
\norm{P_T \mathcal R_n^T P_Tf}_{L^2(\R)}\leq \frac{17 T^{5/2}}{n^{1/2}}\int_{-T}^T|f(y)|\,\mbox{d}y.
$$
\end{remark}

\subsection{Approximating almost time and band limited functions by scaled Hermite functions}\label{sec:scale}
\rep{Inverted the way we scale here to be coherent with the next section}

For $\alpha>0$ and $f\in L^2(\R)$ we define the scaling operator $\delta_\alpha f(x)=\alpha^{-1/2}f(x/\alpha)$.
Recall that $\norm{\delta_\alpha f}_{L^2(\R)}=\norm{f}_{L^2(\R)}$ while
$$
\norm{\delta_\alpha f}_{L^2([-A,A])}=\norm{f}_{L^2([-A/\alpha ,A/\alpha A])},\ 
\norm{\delta_\alpha f}_{L^2(\R\setminus[-A,A])}=\norm{f}_{L^2(\R\setminus[-A/\alpha,A/\alpha])}
$$
and $\ff[\delta_\alpha f]=\delta_{1/\alpha}\ff[f]$.
In particular, if $f$ is $\eps_T$-almost time limited to $[-T,T]$ (resp. $\eps_\Omega$-almost band limited to $[-\Omega,\Omega]$) then
$\delta_\alpha f$ is $\eps_T$-almost time limited to $[-T/\alpha,T/\alpha]$

Next, define the scaled Hermite basis $h_k^\alpha=\delta_{\alpha}h_k$
which is also an orthonormal basis of $L^2(\R)$ and define the corresponding orthogonal projections:
for $f\in L^2(\R)$,
\begin{equation}
\label{scaled_approx}
K^\alpha_nf=\sum_{k=0}^n\scal{f,h_k^\alpha}h_k^\alpha.
\end{equation}

\begin{proposition}\label{th:scaled}
Let $\alpha>0$, $T\geq2$ and $c\geq 2/\alpha$.
Assume that and
$$
\int_{|t|>T}|f(t)|^2\,\mathrm{d}t\leq \eps_T^2\norm{f}_{L^2(\R)}^2
\quad\mbox{and}\quad
\int_{|\omega|>c/\alpha}|\widehat{f}(\omega)|^2\,\mathrm{d}\omega\leq \eps_{c/\alpha}^2\norm{f}_{L^2(\R)}^2.
$$
Then, for $n\geq \max(2(T/\alpha)^2 ,2c^2)$, we have
\begin{equation}
\label{approx_scaled}
\norm{f-K_n^\alpha f}_{L^2([-T,T])}\leq\left(\eps_T+\eps_{c/\alpha}+\frac{24(T/\alpha)^3}{\sqrt{2n+1}}\right)\norm{f}_{L^2(\R)}.
\end{equation}
\end{proposition}

\begin{remark}
The scaling with $\alpha>1$ has as effect to decrease the dependence on $T$ at the price of increasing
the dependence on good frequency concentration, while taking $\alpha<1$ the gain and loss are reversed.
In practice, the above dependence on $T$ is very pessimistic and $\alpha>1$ is a better choice. The most natural choice is $\alpha=T$
and $c=T\Omega$ where $\Omega$ is such that $f$ is $\eps_\Omega$-almost band limited to $[-\Omega,\Omega]$.
\end{remark}

\begin{proof} For $f\in L^2(\R)$, since $K_n^\alpha$ is contractive, we have
\begin{eqnarray*}
\norm{f-K_n^\alpha f}_{L^2([-T,T])}&\leq& \norm{f-K_n^\alpha P_T f}_{L^2([-T,T])}+\norm{K_n^\alpha(f- P_T f)}_{L^2([-T,T])}\\
&\leq& \norm{f-K_n^\alpha P_T f}_{L^2([-T,T])}+\norm{f- P_T f}_{L^2([-T,T])}\\
&\leq&\norm{f-K_n^\alpha P_T f}_{L^2([-T,T])}+\eps_T\norm{f}_{L^2(\R)}.
\end{eqnarray*}
Moreover, $K_n^\alpha P_T f(x)$ is
\begin{eqnarray*}
&=&\sum_{k=0}^n\scal{P_T f,h_k^\alpha}h_k^\alpha(x)=\int_{-T}^T
f(y)\frac{1}{\alpha}\sum_{k=0}^nh_k(x/\alpha)h_k(y/\alpha)\,\mbox{d}y\\
&=&\int_{-T/\alpha}^{T/\alpha}f(\alpha t)\sum_{k=0}^nh_k(x/\alpha)h_k(t)\,\mbox{d}t.
\end{eqnarray*}
Therefore $\norm{f-K_n^\alpha P_T f}_{L^2([-T,T])}$ is
\begin{eqnarray*}
&=&
\left(\int_{-T}^T\abs{f(x)-\int_{-T/\alpha}^{T/\alpha}f(\alpha t)\sum_{k=0}^nh_k(x/\alpha)h_k(t)\,\mbox{d}t}^2
\,\mbox{d}x\right)^{1/2}\\
&=&\left(\int_{-T/\alpha}^{T/\alpha}\abs{\alpha^{1/2}f(\alpha s)
-\int_{-T/\alpha}^{T/\alpha}f(\alpha t)\sum_{k=0}^nh_k(x/\alpha)h_k(t)\,\mbox{d}t}^2
\,\mbox{d}s\right)^{1/2}\\
&=&\norm{f_\alpha-K_nf_\alpha}_{L^2([-\alpha T,\alpha T])}
\end{eqnarray*}
where $f_\alpha=\delta_{1/\alpha}\bigl[\mathbf{1}_{[-T,T]}f\bigr]$.
Note that $f_\alpha$ is $0$-almost time limited to $[-T/\alpha,T/\alpha]$. Next, writing
$$
\widehat{f_\alpha}=\delta_{\alpha}\ff[\mathbf{1}_{[-T,T]}f]=\delta_{\alpha}\ff[f]
-\delta_{\alpha}\ff[\mathbf{1}_{\R\setminus [-T,T]}f]
$$
and, noting that
$$
\norm{\delta_{\alpha}\ff[f]}_{L^2(\R\setminus[-c,c])}=\norm{\ff[f]}_{L^2(\R\setminus[-c/\alpha,c/\alpha])}
\leq\eps_{c/\alpha}\norm{f}_{L^2(\R)}
$$
while
\begin{eqnarray*}
\norm{\delta_{\alpha}\ff[\mathbf{1}_{\R\setminus [-T,T]}f]}_{L^2(\R\setminus[-\Omega,\Omega])}
&\leq&\norm{\delta_{\alpha}\ff[\mathbf{1}_{\R\setminus [-T,T]}f]}_{L^2(\R)}\\
&=&\norm{\mathbf{1}_{\R\setminus [-T,T]}f}_{L^2(\R)}\leq\eps_T\norm{f}_{L^2(\R)},
\end{eqnarray*}
we get
$$
\norm{\widehat{f_\alpha}}_{L^2(\R\setminus[-c,c])}
\leq \eps_{c/\alpha}\norm{f}_{L^2(\R)}+\eps_T\norm{f}_{L^2(\R)}.
$$
It remains to apply Theorem \ref{th:projhermite} to complete the proof.
\end{proof}

\section{Numerical results}

In this paragraph, we give several examples that illustrate the different results of this work.\\

\noindent
{\bf Example 1.} In this example, we check numerically that the 
approximation error $\dst E(x,y)=\abs{\sum_{k=0}^n h_k(x) h_k(y)-\frac{\sin N(x-y)}{\pi(x-y)}}$
is much smaller than the theoretical error given by Theorem \ref{th:estkn}.
In order to do so, we consider a uniform discretization $\Lambda$ of the square $[-1,1]^2$ with 6400 equidistant nodes.
We then estimate the uniform approximation error $\sup|E(x,y|)$ by
$\dst \widetilde E_n=\sup_{x,y\in \Lambda}|E(x,y)|$ and the 
Hilbert-Schmidt norm $\|\mathcal R_n^T\|_{HS}$ that appears in \eqref{eq:approxhermloc}
for $10 \leq n\leq 100$.
\rep{Modified paragraph}
\rea{New second row for the table}
\begin{center}
\begin{tabular}{c|c|c|c|c|c}
$n$&10&25&50&75&100\\
\hline
$\widetilde E_n$& 0.067& 0.039& 0.025& 0.023& 0.022\\
\hline
$\widetilde R_n$& 0.051 &0.034& 0.022 & 0.019& 0.017\\
\end{tabular}
\end{center}

\medskip

\noindent{\bf Example 2.} In this example, we illustrate the quality of approximation by scaled Hermite functions of a time limited and an almost band limited 
function. For this purpose, we consider the function $f(x)= \mathbf{1}_{[-1/2,1/2]}(x)$. From the Fourier transform of $f,$ one can easily check that 
$f\in H^s(\mathbb R)$ for any $s< 1/2.$ Note that $f$ is $0-$concentrated in $[-1/2,1/2]$ and since $f\in H^s(\mathbb R),$ then 
$\epsilon_\Omega$-band concentrated in $[-\Omega, +\Omega]$,for any $\epsilon_\Omega < M_s \Omega^{-s}$ 
with $M_s$ a positive constant. We have 
considered the value $\alpha=10$ and we have used \eqref{scaled_approx} to compute the scaled Hermite approximations $K_n^\alpha(f)$ of $f$ with 
$n=40$ and $n=80$. The graphs of $f$ and its 
scaled Hermite approximation are given by Figure \ref{fig1}. In Figure 2, we have given the approximation errors $f(x)-K_n^\alpha f(x).$\\

 \begin{figure}[ht!]
\includegraphics[width=13.5cm]{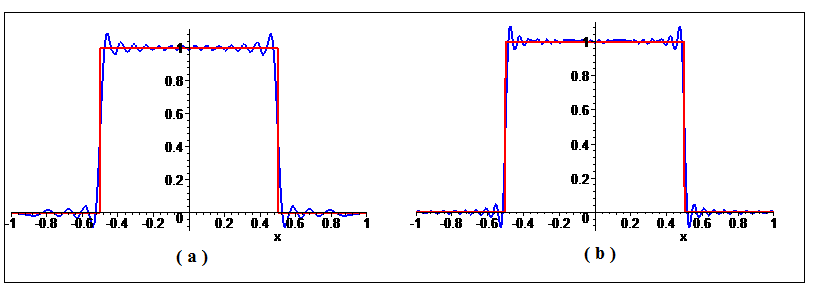}
\caption{The Graph of $f(x)= \mathbf{1}_{[-1/2,1/2]}(x)$ (red) and of $K_n^{c} f(x)$ (blue)
with (a) $n=40$, $\alpha=100$ and (b) $n=80$, $\alpha=10$. Note the Gibbs phenomena that appears.}
\label{fig1}
\end{figure}

 \begin{figure}[ht!]
\includegraphics[width=13.5cm]{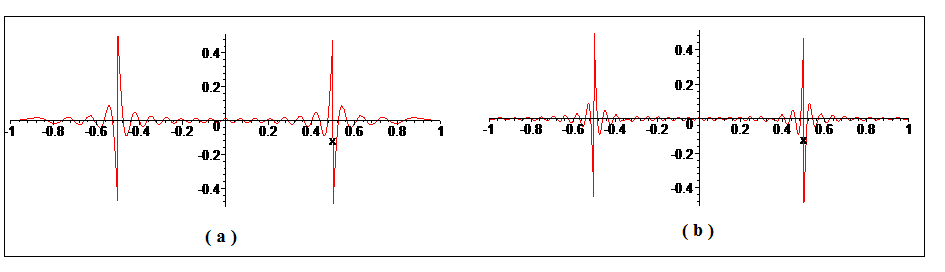}
 \caption{ Graph of the approximation error $f(x)-K_n^{\alpha} f(x)$, $\alpha=10$, (a)  $n=40$  (b) 
	$n=80$.}
	\label{fig2}
\end{figure}

Also, to illustrate the fact that the scaled Hermite approximation outperforms the usual Hermite approximation, we have repeated the previous numerical tests without the scaling factor ({\it i.e.} with $\alpha=1$). 
Figure \ref{fig3} shows the graphs of $f$ and $K_n f$.

 \begin{figure}[ht!]
\includegraphics[width=13.5cm]{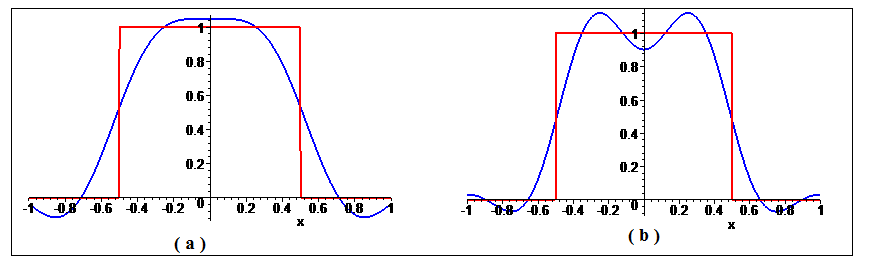}
\caption{The Graph of $f(x)= \mathbf{1}_{[-1/2,1/2]}(x)$ (red) and of $K_n^{\alpha} f(x)$ (blue)
with (a) $n=40$, $\alpha=1$ and (b) $n=80$, $\alpha=1$.}
\label{fig3}
\end{figure} 

\medskip

\noindent
{\bf Example 3.} In this  last example, we illustrate the quality of approximation of almost band limited and time limited function by 
the scaled Hermite functions for the function $g$ given by $g(x) = (1-|x|)\chi_{[-1,1]}(x)$. As is easily seen by expressing the Fourier transform of $g$, $g\in H^s(\mathbf R)$ for any $s<3/2$.
Moreover since $g$ is suppoted on $[-1,1]$,  $g$ is $0$-concentrated in $[-1, 1]$. Moreover, as in the previous example
$g$ is $\epsilon_\Omega$-band concentrated in 
$[-\Omega, +\Omega]$,
for any $\epsilon_\Omega < M_s \Omega^{-s}.$  
We have considered the four  couples 
$(\alpha,n)=(\sqrt{10},20), (\sqrt{10},50), (\sqrt{50},20), (\sqrt{50},50)$ and  computed $K_n^{\alpha} g$.
The numerical results are given by Figures \ref{fig5} and \ref{fig4}. These numerical results suggest again  that the scaled Hermite functions are
well suited for the approximation of almost band limited and almost time limited functions. In this sense, they have similar 
approximation properties as the PSWFs. The actual approximation error is much smaller than the theoretical error given by Theorem 
\ref{th:scaled}
This actual approximation error depends on the truncation order $n$ as well as on the parameter $\alpha$.  

 \begin{figure}[ht!]
  {\includegraphics[width=13.5cm]{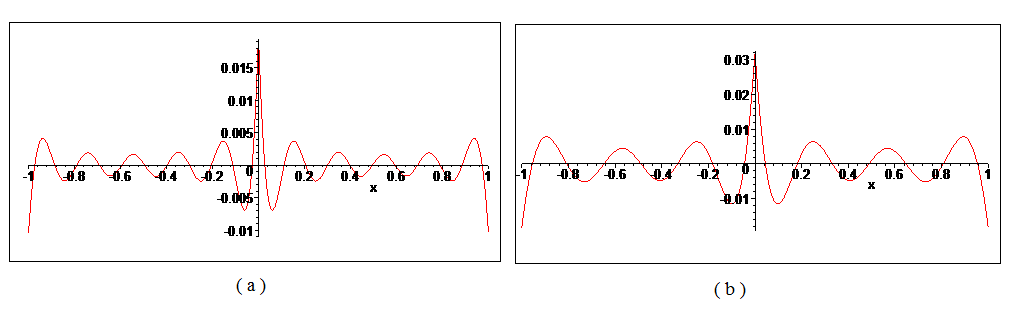}}
  \caption{Graph of the approximation error $g(x)-K_n^{c} g(x)$ with (a) $c=10$, $n=20$ and (b)	$c=10$, $n=50$.}
	\label{fig5}
  \end{figure}

 \begin{figure}[!ht]
  {\includegraphics[width=13.5cm]{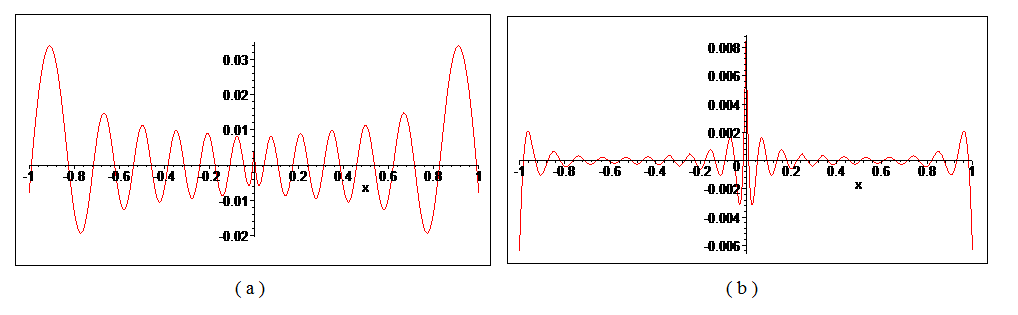}}
  \caption{Graph of the approximation error $g(x)-K_n^{c} g(x)$ with (a) $c=50$, $n=20$ and (b)	$c=50$, $n=50$.}
	\label{fig4}
  \end{figure}

\section*{Acknowledgements}
The first author kindly acknowledge financial support from the French ANR programs ANR
2011 BS01 007 01 (GeMeCod), ANR-12-BS01-0001 (Aventures).
This study has been carried out with financial support from the French State, managed
by the French National Research Agency (ANR) in the frame of the ”Investments for
the future” Programme IdEx Bordeaux - CPU (ANR-10-IDEX-03-02).

\end{document}